\documentclass[reqno]{amsart}

\usepackage{amssymb}
\usepackage{a4wide}

\newtheorem{theorem}{Theorem}

\newtheorem{proposition}[theorem]{Proposition}
\newtheorem{lemma}[theorem]{Lemma}
\newtheorem{definition}[theorem]{Definition}
\newtheorem{example}[theorem]{Example}
\newtheorem{remark}[theorem]{Remark}
\newtheorem{notation}[theorem]{Notation}

\AtBeginDocument{{\noindent\small
This is a preprint of a paper whose final and definite form will be published
in \emph{Journal of Computational and Applied Mathematics}, ISSN: 0377-0427.
Paper Submitted 04/Jul/2015; Revised 14/Dec/2015 and 03/Jan/2016; 
Accepted for publication 08/Jan/2016.}
\vspace{9mm}}

\begin{document}

\author[B. Bayour]{Benaoumeur Bayour}
\address{Benaoumeur Bayour\newline
\indent University of Chlef, B.P. 151, Hay Es-salem, Chlef, Algeria}
\email{b.benaoumeur@gmail.com}


\author[D. F. M. Torres]{Delfim F. M. Torres}
\address{Delfim F. M. Torres\newline
\indent Center for Research and Development in Mathematics and Applications (CIDMA),\newline
\indent Department of Mathematics, University of Aveiro, 3810-193 Aveiro, Portugal}
\indent \email{delfim@ua.pt}


\title[Existence of Solution to a Local Fractional Differential Equation]{%
Existence of Solution to a Local Fractional Nonlinear Differential Equation}

\begin{abstract}
We prove existence of solution to a local fractional nonlinear
differential equation with initial condition. For that
we introduce the notion of tube solution.
\end{abstract}

\subjclass[2010]{26A33, 34A12}

\keywords{Existence of solutions, fractional differential equations,
local fractional derivatives, conformable fractional derivatives,
initial value problems.}

\maketitle


\section{Introduction}

Fractional calculus is a branch of mathematical analysis that
studies the possibility of taking noninteger order powers
of the differentiation and/or integration operators.
Even though the term ``fractional'' is a misnomer,
it has been widely accepted for a long time:
the term was coined by the famous mathematician Leibniz in 1695
in a letter to L'Hopital \cite{MR1347689}. In the paper
\emph{What is a fractional derivative?} \cite{MR3342452},
Ortigueira and Machado distinguish between local and nonlocal
fractional derivatives. Here we are concerned with local operators only.
Such local approach to the fractional calculus dates back at least to 1974,
to the use of the fractional incremental ratio in \cite{MR0492115}.
For an overview and recent developments of the local approach
to fractional calculus we refer the reader to \cite{Magin,MR3351489,Yang,MR3204525}
and references therein.

Recently, Khalil et al. introduced a new well-behaved definition of local fractional
(noninteger order) derivative, called the conformable fractional derivative \cite{khs}.
The new calculus is very interesting and is getting an increasing of interest
-- see \cite{MyID:324,Chung} and references therein. In \cite{MR3293309}, Abdeljawad
proves chain rules, exponential functions, Gronwall's inequality,
fractional integration by parts, Taylor power series expansions
and Laplace transforms for the conformable fractional calculus.
Furthermore, linear differential systems are discussed \cite{MR3293309}.
In \cite{MR3326681}, Batarfi et al. obtain the Green function for a conformable
fractional linear problem and then introduce the study of nonlinear
conformable fractional differential equations. See also \cite{MR3335759}
where, using the conformable fractional derivative, a second-order conjugate
boundary value problem is investigated and utilizing the corresponding
positive fractional Green's function and an appropriate fixed point theorem,
existence of a positive solution is proved. For abstract Cauchy problems
of conformable fractional systems see \cite{ref:ref2}. Here we are concerned
with the following problem:
\begin{equation}
\label{pl}
\begin{cases}
x^{(\alpha)}(t)=f(t,x(t)), & t\in[a,b], \quad a>0,\\
x(a)=x_{0},\\
\end{cases}
\end{equation}
where $f:[a,b]\times \mathbb{R}\rightarrow \mathbb{R}$ is a continuous function,
$x^{(\alpha)}(t)$ denotes the conformable fractional derivative of $x$ at $t$
of order $\alpha$, $\alpha\in(0,1)$. For the first time in the literature
of conformable fractional calculus, we introduce the notion of tube solution.
Such idea of tube solution has been used with success to investigate existence
of solutions for ordinary differentiable equations \cite{MF,HGF}, delta and
nabla differential equations on time scales \cite{MyID:332,HG2,HG},
and dynamic inclusions \cite{HGF2}. Roughly speaking, the tube solution method
generalizes the method of lower and upper solution \cite{MBCT,ACG,JGL,MRZJ}.

The paper is organized as follows. In Section~\ref{sec:2}, we present
the main concepts of the local conformable fractional calculus and we give some
useful preliminary results. In Section~\ref{sec:3}, we prove existence
of solution to problem \eqref{pl} by using the notion of tube solution
and Schauder's fixed-point theorem (see Theorem~\ref{thm:mr}).
We end with Section~\ref{sec:4}, where an illustrative example is given.


\section{Preliminaries}
\label{sec:2}

We consider fractional derivatives in the conformable sense \cite{khs}.

\begin{definition}[Conformable fractional derivative \cite{khs}]
\label{def:cfd}
Let $\alpha\in(0,1)$ and $f : [0,\infty)\rightarrow \mathbb{R}$. The
conformable fractional derivative of $f$ of order $\alpha$ is defined by
$T_{\alpha}(f)(t):=\lim_{\epsilon\rightarrow0}
\frac{f(t+\epsilon t^{1-\alpha})-f(t)}{\epsilon}$
for all $t>0$. Often, we write $f^{(\alpha)}$ instead of $T_{\alpha}(f)$
to denote the conformable fractional derivative of $f$ of order $\alpha$.
In addition, if the conformable fractional derivative of $f$ of order
$\alpha$ exists, then we simply say that $f$ is $\alpha$-differentiable.
If $f$ is $\alpha$-differentiable in some $t \in (0,a)$, $a>0$, and
$\lim_{t\rightarrow 0^{+}}f^{(\alpha)}(t)$ exists, then we
define $f^{(\alpha)}(0):=\lim_{t\rightarrow 0^{+}}f^{(\alpha)}(t)$.
\end{definition}

\begin{theorem}[\cite{khs}]
\label{th0}
Let $\alpha\in(0,1]$ and assume $f,g$ to be $\alpha$-differentiable. Then,
\begin{enumerate}
\item $T_{\alpha}(af+bg)=aT_{\alpha}(f)+bT_{\alpha}(g)$ for all $a,b\in\mathbb{R}$;
\item $T_{\alpha}(fg)=fT_{\alpha}(g)+gT_{\alpha}(f)$;
\item $T_{\alpha}\left(f/g\right)=\left(gT_{\alpha}(f)-fT_{\alpha}(g)\right)/g^{2}$.
\end{enumerate}
If, in addition, $f$ is differentiable at a point $t>0$, then
$T_{\alpha}(f)(t)=t^{1-\alpha}\frac{df}{dt}(t)$.
\end{theorem}

\begin{remark}
From Theorem~\ref{th0} it follows that if $f \in C^1$, then one has
$$
\lim_{\alpha\rightarrow 1} T_\alpha(f)(t) = f'(t)
$$
and
\begin{equation}
\label{eq:der:0}
\lim_{\alpha\rightarrow 0} T_\alpha(f)(t) = t f'(t).
\end{equation}
So $T_\alpha(f)$ is ``conformable'' in the sense it coincides
with $f'$ in the case $\alpha\rightarrow 1$ and satisfies
similar properties to the integer-order calculus.
Note that the property $\lim_{\alpha\rightarrow 0} T_\alpha(f) \neq f$
is not uncommon in fractional calculus, both for local and nonlocal
operators: see, e.g., the local fractional derivative of
\cite{Katugampola,Katugampola:arXiv}, for which property
\eqref{eq:der:0} also holds \cite{Anderson};
and the classical nonlocal Marchaud fractional derivative,
which is zero when $\alpha \rightarrow 0$ \cite{MR1347689}.
Note, however, that we only have $T_\alpha(f)(t) = t^{1-\alpha} f'(t)$
in case $f$ is differentiable. If one considers
a function that is not differentiable at a point $t$,
then the conformable derivative is not $t^{1-\alpha} f'(t)$.
For applications we refer the reader to \cite{Chung}.
\end{remark}

\begin{example}
\label{ex}
Let $\alpha\in(0,1]$. Functions $f(t) = t^{p}$, $p\in\mathbb{R}$,
$g(t) \equiv \lambda$, $\lambda \in \mathbb{R}$,
$h(t) = e^{ct}$, $c\in\mathbb{R}$,
and $\beta(t) = e^{\frac{1}{\alpha}t^{\alpha}}$,
are $\alpha$-differentiable with conformable
fractional derivatives of order $\alpha$ given by
\begin{enumerate}
\item $T_{\alpha}(f)(t)=pt^{p-\alpha}$;
\item $T_{\alpha}(g)(t)=0$;
\item $T_{\alpha}(h)(t)=ct^{1-\alpha}e^{ct}$;
\item $T_{\alpha}(\beta)(t)=e^{\frac{1}{\alpha}t^{\alpha}}$.
\end{enumerate}
\end{example}

\begin{remark}
Differentiability implies $\alpha$-differentiability but the contrary
is not true: a nondifferentiable function can be $\alpha$-differentiable.
For a discussion of this issue see \cite{khs}.
\end{remark}

\begin{definition}[Conformable fractional integral \cite{khs}]
Let $\alpha\in(0,1)$ and $f : [a,\infty)\rightarrow \mathbb{R}$. The
conformable fractional integral of $f$ of order $\alpha$
from $a$ to $t$, denoted by $I_{\alpha}^{a}(f)(t)$, is defined by
$$
I_{\alpha}^{a}(f)(t)
:=\int_{a}^{t}\frac{f(\tau)}{\tau^{1-\alpha}}d\tau,
$$
where the above integral is the usual improper Riemann integral.
\end{definition}

\begin{theorem}[\cite{khs}]
\label{th1}
If $f$ is a continuous function in the domain of $I_{\alpha}^{a}$, then
$T_{\alpha}\left(I_{\alpha}^{a}(f)\right)(t)=f(t)$ for all $t\geq a$.
\end{theorem}

\begin{notation}
Let $0 < a < b$. We denote by ${_{\alpha}\mathfrak{J}}_{a}^{b}[f]$
the value of the integral $\int_{a}^{b}\frac{f(t)}{t^{1-\alpha}}dt$,
that is, ${_{\alpha}\mathfrak{J}}_{a}^{b}[f] := I_{\alpha}^{a}(f)(b)$.
\end{notation}

\begin{proposition}
\label{pr1}
Assume $f\in L^{1}([a,b],\mathbb{R})$, $0<a<b$. Then
$\left|{_{\alpha}\mathfrak{J}}_{a}^{b}[f]\right|
\leq  {_{\alpha}\mathfrak{J}}_{a}^{b}[|f|]$.
\end{proposition}

\begin{proof}
Let $f\in L^{1}([a,b],\mathbb{R})$. Then,
$$
\left|{_{\alpha}\mathfrak{J}}_{a}^{b}[f]\right|
=\left|\int_{a}^{b}\frac{f(t)}{t^{1-\alpha}}dt\right|
\leq \int_{a}^{b} \left|\frac{f(t)}{t^{1-\alpha}}\right| dt
=\int_{a}^{b} \frac{|f(t)|}{t^{1-\alpha}}dt.
$$
Therefore, $\left|{_{\alpha}\mathfrak{J}}_{a}^{b}[f]\right|
\leq  {_{\alpha}\mathfrak{J}}_{a}^{b}[|f|]$
and the proposition is proved.
\end{proof}

\begin{notation}
We denote by $C^{(\alpha)}([a,b],\mathbb{R})$, $0<a<b$, $\alpha>0$,
the set of all real-valued functions $f:[a,b]\rightarrow \mathbb{R}$ that are
$\alpha$-differentiable and for which the $\alpha$-derivative is continuous.
We often abbreviate $C^{(\alpha)}([a,b],\mathbb{R})$ by $C^{(\alpha)}([a,b])$.
\end{notation}

The next lemma is a consequence of the conformable mean
value theorem proved in \cite{khs} by noting the discussion under
Definition 2.1 in \cite{MR3293309}. Note that
$r(b) - r(a) = I_\alpha^a\left(r^{(\alpha)}\right)(b)$
follows from Lemma~2.8 in \cite{MR3293309}.

\begin{lemma}
\label{le}
Let $r\in C^{(\alpha)}([a,b])$, $0<a<b$, such that $r^{(\alpha)}(t)<0$ on
$\left\{ t\in[a,b] : r(t)>0\right\}$. If $r(a)\leq 0$,
then $r(t)\leq 0$ for every $t\in[a,b]$.
\end{lemma}

\begin{proof}
Suppose the contrary. If there exists
$t\in[a,b]$ such that $r(t)>0$, then there exists
$t_{\circ}\in[a,b]$ such that $r(t_{\circ})=\max_{a\leq t\leq b}(r(t))>0$
because $r\in C^{(\alpha)}([a,b])$ and $r(t)>0$. There are two cases.
(i) if $t_{\circ}>a$, then there exists an interval
$[t_{1},t_{\circ}]$ included in $[a, t_{\circ}]$ such that $r(t)>0$
for all $t\in[t_{1},t_{\circ}]$. It follows from the assumption
$r^{(\alpha)}(t)<0$ for all $t\in[t_{1},t_{\circ}]$ and Lemma~2.8
of \cite{MR3293309} that $I_{\alpha}^{t_{1}}\left(r^{(\alpha)}\right)(t_{\circ})
=r(t_{\circ})-r(t_{1})<0$, which contradicts the fact that
$r(t_{\circ})$ is a maximum. (ii) If $t_{\circ}=a$, then $r(t_{\circ})>0$
is impossible from hypothesis.
\end{proof}

\begin{theorem}
\label{th2}
If $g\in L^{1}([a,b])$, then function $x:[a,b]\rightarrow \mathbb{R}$ defined by
\begin{equation}
\label{eq:f:x:th2}
x(t):=e^{-\frac{1}{\alpha}\left(\frac{t}{a}\right)^{\alpha}}\left(
e^{\frac{1}{\alpha}}x_{0} + {_{\alpha}\mathfrak{J}}_{a}^{t}\left[
\frac{g(s)}{e^{-\frac{1}{\alpha}(\frac{s}{a})^{\alpha}}}\right]\right)
\end{equation}
is solution to problem
\begin{equation*}
\begin{cases}
x^{(\alpha)}(t)+\frac{1}{a^{\alpha}}x(t)=g(t), & t\in[a,b], \quad a>0,\\
x(a)=x_{0}.
\end{cases}
\end{equation*}
\end{theorem}

\begin{proof}
Let $x:[a,b]\rightarrow \mathbb{R}$ be the function defined by \eqref{eq:f:x:th2}.
We know from Theorems~\ref{th0} and \ref{th1} that
\begin{equation*}
\begin{split}
x^{(\alpha)}(t)&= t^{1-\alpha}\left(-\frac{1}{\alpha}\left(
\frac{1}{a}\right)^{\alpha}\alpha t^{\alpha-1}\right)e^{
-\frac{1}{\alpha}\left(\frac{t}{a}\right)^{\alpha}}\left(
e^{\frac{1}{\alpha}}x_{0}+{_{\alpha}\mathfrak{J}_{a}^{t}}\left[\frac{g(s)}{e^{
-\frac{1}{\alpha}(\frac{s}{a})^{\alpha}}}\right]\right)
+ e^{-\frac{1}{\alpha}(\frac{t}{a})^{\alpha}}\left(
\frac{g(t)}{e^{-\frac{1}{\alpha}(\frac{t}{a})^{\alpha}}}\right)\\
&=-\left(\frac{1}{a}\right)^{\alpha}
e^{-\frac{1}{\alpha}(\frac{t}{a})^{\alpha}}\left(
e^{\frac{1}{\alpha}}x_{0}+{_{\alpha}\mathfrak{J}}_{a}^{t}\left[
\frac{g(s)}{e^{-\frac{1}{\alpha}\left(\frac{s}{a}\right)^{\alpha}}}\right]\right)
+g(t)\\
&=-\left(\frac{1}{a}\right)^{\alpha} x(t)+g(t).
\end{split}
\end{equation*}
We just obtained that
$x^{(\alpha)}(t)+\left(\frac{1}{a}\right)^{\alpha} x(t)=g(t)$.
On the other hand,
$$
x(a)=e^{-\frac{1}{\alpha}(\frac{a}{a})^{\alpha}}\left(
e^{\frac{1}{\alpha}}x_{0}+{_{\alpha}\mathfrak{J}}_{a}^{a}\left[
\frac{g(s)}{e^{-\frac{1}{\alpha}\left(\frac{s}{a}\right)^{\alpha}}}\right]\right)
=e^{-\frac{1}{\alpha}}\left(e^{\frac{1}{\alpha}}x_{0}+0\right)
= x_{0}
$$
and the proof is complete.
\end{proof}

Theorem~\ref{th2} is enough for our purposes.
It should be mentioned, however, that it can be generalized
by benefiting from Lemma~2.8 in \cite{MR3293309} with its
higher-order version \cite[Proposition~2.9]{MR3293309}.

\begin{theorem}
\label{thm:ref2}
If $g\in L^{1}([a,b])$ and $p(t)$ is continuous on $[a,b]$,
then the function $x:[a,b]\rightarrow \mathbb{R}$ defined by
\begin{equation}
\label{eq:ref2:conc}
x(t)=\frac{1}{\mu(t)}\left(x(a)\mu(a)+I_{\alpha}^{a}(\mu g)(t)\right)
\end{equation}
is a solution to the linear conformable equation
\begin{equation}
\label{eq:1:ref}
x^{(\alpha)}(t)+p(t)x(t)=g(t),
\quad x(a)=x_0, \quad a>0.
\end{equation}
\end{theorem}

\begin{proof}
Consider the integrating factor function $\mu(t)=e^{I_{\alpha}^{a}(p)(t)}$.
Then, by means of item (3) of Example~\ref{ex} and the Chain Rule
\cite[Theorem~2.11]{MR3293309}, one can see that $\mu^{(\alpha)}(t)= p(t)\mu(t)$.
Then, multiply \eqref{eq:1:ref} by function $\mu(t)$. By means of the product
rule (item (2) of Theorem~\ref{th0}), \eqref{eq:1:ref} turns to
\begin{equation}
\label{eq:2:ref}
(x(t)\mu(t))^{(\alpha)}=\mu(t)g(t).
\end{equation}
Apply $I_{\alpha}^{a}$ to \eqref{eq:2:ref} and use Lemma~2.8 in \cite{MR3293309}
to conclude that \eqref{eq:ref2:conc} holds.
\end{proof}

Theorem~\ref{th2} follows as a corollary
from Theorem~\ref{thm:ref2} by putting
$\mu(t)=e^{\frac{1}{\alpha}(\frac{t}{a})^{\alpha}}$
and $p(t)=e^{\frac{1}{\alpha}}$.

\begin{proposition}
\label{pr3}
If $x:(0,\infty)\rightarrow \mathbb{R}$ is $\alpha$-differentiable at $t\in[a,b]$,
then $|x(t)|^{(\alpha)}=\frac{x(t) \, x^{\alpha}(t)}{|x(t)|}$.
\end{proposition}

\begin{proof}
From Definition~\ref{def:cfd} we have
\begin{equation*}
\begin{split}
\left|x(t)\right|^{(\alpha)}
&=\lim_{\epsilon\rightarrow 0}\frac{|x(t+\epsilon t^{1-\alpha})|
-|x(t)|}{\epsilon}\\
&=\lim_{\epsilon\rightarrow 0}\frac{x\left(t+\epsilon t^{1-\alpha}\right)^{2}
-x(t)^{2}}{\epsilon\left(|x\left(t+\epsilon t^{1-\alpha}\right)|+|x(t)|\right)}\\
&=\lim_{\epsilon\rightarrow 0}\left[\frac{x(t+\epsilon t^{1-\alpha})^{2}
-x(t)^{2}}{\epsilon} \cdot \frac{1}{|x(t+\epsilon t^{1-\alpha})|+|x(t)|}\right]\\
&=\left[x(t)^{2}\right]^{(\alpha)}\frac{1}{2|x(t)|}\\
&=2 x(t) x^{(\alpha)}(t)\frac{1}{2|x(t)|},
\end{split}
\end{equation*}
which proves the intended relation.
\end{proof}


\section{Main Result}
\label{sec:3}

We begin by introducing the notion of tube solution to problem \eqref{pl}.

\begin{definition}
\label{definition}
Let $(v,M)\in C^{(\alpha)}([a,b],\mathbb{R})\times C^{(\alpha)}([a,b],[0,\infty))$.
We say that $(v,M) $ is a tube solution to problem \eqref{pl} if
\begin{enumerate}
\item[(i)] $\left(y-v(t)\right) \left(f(t,y)-v^{(\alpha)}\right)
\leq M(t)M^{(\alpha)}(t)$ for every $t\in[a,b]$
and every $y\in \mathbb{R}$ such that $|y-v(t)|=M(t)$;

\item[(ii)] $v^{(\alpha)}(t)=f(t,v(t))$ and $M^{(\alpha)}(t)=0$
for all $t\in[a,b]$ such that $M(t)=0$;

\item[(iii)] $|x_{0}-v(a)|\leq M(a)$.
\end{enumerate}
\end{definition}

\begin{notation}
We introduce the following notation:
$$
\mathbf{T}(v,M) := \left\{x\in C^{(\alpha)}([a,b],\mathbb{R}) :
| x(t)-v(t)| \leq M(t), \  t\in [a,b]\right\}.
$$
\end{notation}

Consider the following problem:
\begin{equation}
\label{eq:probAux}
\begin{cases}
x^{(\alpha)}+\frac{1}{a^{\alpha}}x(t)
=f(t,\widetilde{x}(t))+\frac{1}{a^{\alpha}}\widetilde{x}(t),
& t\in[a,b], \quad a>0,\\
x(a)=x_{0},
\end{cases}
\end{equation}
where
\begin{equation}
\label{eq:probAux:wt}
\widetilde{x}(t) :=
\begin{cases}
\frac{M(t)}{|x-v(t)|}(x(t)-v(t))+v(t)& \text{ if } |x-v(t)|> M(t),\\
x(t) & \text{ otherwise}.
\end{cases}
\end{equation}
Let us define the operator $\mathbf{N}:C([a,b])\rightarrow C([a,b])$ by
$$
\mathbf{N}(x)(t) := e^{-\frac{1}{\alpha}(\frac{t}{a})^{\alpha}}\left(
e^{\frac{1}{\alpha}}x_{0}+ {_{\alpha}\mathfrak{J}}_{a}^{t}\left[
\frac{f(s,\widetilde{x}(s))+\frac{1}{a^{\alpha}}\widetilde{x}(s)}{
e^{-\frac{1}{\alpha}(\frac{s}{a})^{\alpha}}}\right]\right).
$$
In the proof of Proposition~\ref{pr2}
we use the concept of compact function.

\begin{definition}[See p.~112 of \cite{MR1987179}]
Let $X$, $Y$ be topological spaces.
A map $f:X\rightarrow Y$ is called compact if $f(X)$
is contained in a compact subset of $Y$.
\end{definition}

Compact operators occur in many problems of classical analysis.
Note that operator $N$ is nonlinear
because $f$ is nonlinear. In the nonlinear case, the first
comprehensive research on compact operators was due to Schauder
\cite[p.~137]{MR1987179}. In this context,
the Arzel\`{a}--Ascoli theorem asserts that a subset is relatively compact
if and only if it is bounded and equicontinuous
\cite[p.~607]{MR1987179}.

\begin{proposition}
\label{pr2}
If $(v,M)\in C^{(\alpha)}([a,b],\mathbb{R})\times C^{(\alpha)}([a,b],[0,\infty))$
is a tube solution to \eqref{pl}, then
$\mathbf{N}:C([a,b])\rightarrow C([a,b])$ is compact.
\end{proposition}

\begin{proof}
Let $\{x_{n} \}_{n\in\mathbb{N}}$ be a sequence of $C([a,b],\mathbb{R})$
converging to $x\in C([a,b],\mathbb{R})$. By Proposition~\ref{pr1},
\begin{equation*}
\begin{split}
|\mathbf{N}(x_{n}(t)) &- \mathbf{N}(x(t))|
=\Bigg|e^{-\frac{1}{\alpha}(\frac{s}{a})^{\alpha}}\left( e^{\frac{1}{\alpha}}x_{0}
+ {_{\alpha}\mathfrak{J}}_{a}^{t}\left[\frac{f(s,\widetilde{x}_{n}(s))
+ \frac{1}{a^{\alpha}}\widetilde{x}_{n}(s)}{e^{-\frac{1}{\alpha}(
\frac{s}{a})^{\alpha}}}\right]\right)\\
& \qquad -e^{-\frac{1}{\alpha}\left(\frac{s}{a}\right)^{\alpha}}
\left( e^{\frac{1}{\alpha}}x_{0}
+ {_{\alpha}\mathfrak{J}}_{a}^{t}\left[\frac{f(s,\widetilde{x}(s))
+ \frac{1}{a^{\alpha}}\widetilde{x}(s)}{
e^{-\frac{1}{\alpha}\left(\frac{s}{a}\right)^{\alpha}}}\right]\right) \Bigg|\\
&\leq \frac{K}{C} {_{\alpha}\mathfrak{J}}_{a}^{t}\left[\left|
\left(f(s,\widetilde{x}_{n}(s))+ \frac{1}{a^{\alpha}}\widetilde{x}_{n}(s)\right)
- \left(f(s,\widetilde{x}(s)) + \frac{1}{a^{\alpha}}\widetilde{x}(s)\right)
\right|\right]\\
&\leq \frac{K}{C} \left({_{\alpha}\mathfrak{J}}_{a}^{t}
\left[\left|f(s,\widetilde{x}_{n}(s))-f(s,\widetilde{x}(s))\right|\right]
+\frac{1}{a^{\alpha}}{_{\alpha}\mathfrak{J}}_{a}^{t}
\left[\left|\widetilde{x}_{n}(s)-\widetilde{x}(s)\right|\right]\right),
\end{split}
\end{equation*}
where $K:=\max_{a\leq s\leq b}\{e^{-\frac{1}{\alpha}(\frac{s}{a})^{\alpha}} \}$
and $C:=\min_{a\leq s\leq b}\{e^{-\frac{1}{\alpha}(\frac{s}{a})^{\alpha}}\}$.
We need to show that the sequence $\{g_{n}\}_{n\in\mathbb{N}}$ defined by
$g_{n}(s):=f(s,\widetilde{x}_{n}(s))+ \frac{1}{a^{\alpha}}\widetilde{x}_{n}(s)$
converges in $C^{(\alpha)}([a,b])$ to function $g(s)=f(s,\widetilde{x}(s))
+ \frac{1}{a^{\alpha}}\widetilde{x}(s)$. Since there is a constant $R>0$
such that $\|\widetilde{x}\|_{C([a,b],\mathbb{R})}<R$, there exists an index
$N$ such that $\|\widetilde{x}_{n}\|_{C([a,b],\mathbb{R})}\leq R$
for all $n>N$. Thus, $f$ is uniformly continuous on $[a,b]\times B_{R}(0)$.
Therefore, for $\epsilon>0$ given, there is a $\delta>0$ such that
\begin{equation*}
|y-x|<\delta<\frac{C\epsilon \alpha a^{\alpha}}{2k(b^{\alpha}-a^{\alpha})}
\end{equation*}
for all $x,y\in \mathbb{R}$;
\begin{equation*}
|f(s,y)-f(s,x)|<\frac{C\epsilon \alpha }{2k(b^{\alpha}-a^{\alpha})}
\end{equation*}
for all $s\in[a,b]$. By assumption, one can find an index
$\hat{N}>N$ such that
$\|\widetilde{x}_{n}-\widetilde{x}\|_{C([a,b],\mathbb{R})}<\delta$
for $n>\hat{N}$. In this case,
\begin{equation*}
\begin{split}
\left|\mathbf{N}(x_{n})(t)-\mathbf{N}(x)(t)\right|
&<\frac{K}{C}\left( {_{\alpha}\mathfrak{J}}_{a}^{b}\left[\frac{C\epsilon \alpha}{
2k(b^{\alpha}-a^{\alpha})}\right]+\frac{1}{a^{\alpha}}{_{\alpha}\mathfrak{J}}_{a}^{b}
\left[\frac{C\epsilon \alpha a^{\alpha}}{2k(b^{\alpha}-a^{\alpha})}\right]\right)\\
&=\frac{2KC\epsilon \alpha}{2kC(b^{\alpha}-a^{\alpha})}
{_{\alpha}\mathfrak{J}}_{a}^{b}[1]\\
&=\frac{\epsilon \alpha}{b^{\alpha}-a^{\alpha}}
\frac{b^{\alpha}-a^{\alpha}}{\alpha}\\
&=\epsilon.
\end{split}
\end{equation*}
This proves the continuity of $\mathbf{N}$. We now show that the set
$\mathbf{N}(C([a,b]))$ is relatively compact. Consider a sequence
$\{y_{n}\}_{n\in\mathbb{N}}$ of $\mathbf{N}(C([a,b]))$ for all $n\in\mathbb{N}$.
It exists $x_{n}\in C([a,b])$ such that $y_{n}=\mathbf{N}(x_{n})$. Observe
that from Proposition~\ref{pr1} we have
\begin{equation*}
\begin{split}
\left|\mathbf{N}(x_{n})(t)\right|
&=\left|e^{-\frac{1}{\alpha}(\frac{t}{a})^{\alpha}}\left(
e^{\frac{1}{\alpha}}x_{0} + {_{\alpha}\mathfrak{J}}_{a}^{t}\left[
\frac{f(s,\widetilde{x}_{n}(s)) + \frac{1}{a^{\alpha}}\widetilde{x}_{n}(s)}{
e^{-\frac{1}{\alpha}\left(\frac{s}{a}\right)^{\alpha}}}\right]\right)\right|\\
&\leq K\left(e^{\frac{1}{\alpha}}|x_{0}|+\frac{1}{C}
{_{\alpha}\mathfrak{J}}_{a}^{b} \left[\left|f(t,\widetilde{x}_{n}(s))
+ \frac{1}{a^{\alpha}}\widetilde{x}_{n}(s)\right|\right]\right)\\
&\leq K\left(e^{\frac{1}{\alpha}}|x_{0}| + \frac{1}{C}
{_{\alpha}\mathfrak{J}}_{a}^{b}\left[
\left|f\left(t,\widetilde{x}_{n}(s)\right)\right|\right]
+\frac{1}{C a^{\alpha}} {_{\alpha}\mathfrak{J}}_{a}^{b}
\left[\left|\widetilde{x}_{n}(s)\right|\right]\right).
\end{split}
\end{equation*}
By definition, there is an $R>0$ such that $|\widetilde{x}_{n}(s)|\leq R$ for all
$s\in[a,b]$ and all $n\in\mathbb{N}$. The function $f$ is compact
on $[a,b]\times B_{R}(0)$ and we can deduce the existence of a constant $A>0$
such that $|f(s,\widetilde{x}_{n}(s)|\leq A$ for all $s\in[a,b]$.
The sequence $\{y_{n}\}_{n\in\mathbb{N}}$ is uniformly bounded
for all $n\in\mathbb{N}$. Observe also that for $t_{1},t_{2}\in[a,b]$ we have
\begin{equation*}
\begin{split}
| \mathbf{N}(x_{n})(t_{2}) &- \mathbf{N}(x_{n})(t_{1})|\\
&\leq B\left|e^{-\frac{1}{\alpha}(\frac{t_{1}}{a})^{\alpha}}
-e^{-\frac{1}{\alpha}(\frac{t_{2}}{a})^{\alpha}}\right|
+\frac{K(A+\acute{R})}{C}\left|{_{\alpha}\mathfrak{J}}_{t_{1}}^{t_{2}}[1]\right|\\
&< B\left|e^{-\frac{1}{\alpha}(\frac{t_{1}}{a})^{\alpha}}
-e^{-\frac{1}{\alpha}(\frac{t_{2}}{a})^{\alpha}}\right|
+\frac{K(A+\acute{R})}{C}\frac{1}{\alpha}\left|t_{1}^{\alpha}-t_{2}^{\alpha}\right|,
\end{split}
\end{equation*}
where $B:=e^{\alpha}x_{0}$, $\acute{R}:=\frac{R}{a^{\alpha}}$,
$K:=\max_{a\leq t\leq b}\{e^{-\frac{1}{\alpha}(\frac{t}{a})^{\alpha}}\}$,
and $C:=\min_{a\leq t \leq b}\{e^{-\frac{1}{\alpha}(\frac{t}{a})^{\alpha}}\}$.
This proves that the sequence $\{y_{n}\}_{n\in \mathbb{N}}$ is equicontinuous.
By the Arzel\`{a}--Ascoli theorem, $\mathbf{N}(C([a,b]))$ is relatively compact
and hence $\mathbf{N}$ is compact.
\end{proof}

\begin{theorem}
\label{thm:mr}
If $(v,M)\in C^{(\alpha)}([a,b],\mathbb{R})\times C^{(\alpha)}([a,b],[0,\infty))$
is a tube solution to \eqref{pl}, then problem \eqref{pl} has a solution
$x\in C^{(\alpha)}([a,b],\mathbb{R})\cap \mathrm{T}(v,M)$.
\end{theorem}

\begin{proof}
By Proposition~\ref{pr2}, the operator $\mathbf{N}$ is compact.
It has a fixed point by the Schauder fixed point theorem
(see p.~137 of \cite{MR1987179}). Therefore,
Theorem~\ref{th2} implies that such fixed point is a solution
to problem \eqref{eq:probAux}--\eqref{eq:probAux:wt}
and it suffices to show that for every solution $x$
to problem \eqref{eq:probAux}--\eqref{eq:probAux:wt},
$x\in\mathbf{T}(v,M)$. Consider the set
$A :=\left\{t\in[a,b]:|x(t)-v(t)|>M(t)\right\}$.
If $t\in A$, then by virtue of Proposition~\ref{pr3} we have
$$
\left(|x(t)-v(t)|-M(t)\right)^{(\alpha)}
=\frac{\left( x(t)-v(t)\right) \left(x^{(\alpha)}(t)-v^{(\alpha)}(t)
\right)}{|x(t)-v(t)|}-M^{(\alpha)}(t).
$$
Therefore, since $(v,M)$ is a tube solution to problem \eqref{pl},
we have on $\{ t\in A : M(t)> 0\}$ that
\begin{equation*}
\begin{split}
(|x(t)&-v(t)|-M(t))^{(\alpha)}\\
&= \frac{\left(x(t)-v(t)\right) \left(x^{(\alpha)}(t)-v^{(\alpha)}(t)
\right)}{|x(t)-v(t)|} - M^{(\alpha)}(t)\\
&= \frac{\left(x(t)-v(t)\right) \left(f(t,\widetilde{x}(t))
+\left(\frac{1}{a^{\alpha}}\widetilde{x}(t)-\frac{1}{a^{\alpha}}x(t)\right)
-v^{(\alpha)}(t)\right)}{|x(t)-v(t)|}-M^{(\alpha)}(t)\\
&= \frac{\left(\widetilde{x}(t)-v(t)\right)
\left(f(t,\widetilde{x}(t))- v^{(\alpha)}(t)\right)}{M(t)}
+\frac{\left(\widetilde{x}(t)-v(t)\right) \left(\widetilde{x}(t)-x(t)
\right)}{a^{\alpha}M(t)}-M^{(\alpha)}(t)\\
&=\frac{\left(\widetilde{x}(t)-v(t)\right)
\left(f(t,\widetilde{x}(t))- v^{(\alpha)}(t)\right)}{M(t)}
+\left[\frac{M(t)}{|x(t)-v(t)|}-1\right]\frac{|x(t)-v(t)|^{2}}{a^{\alpha}
\left|x(t)-v(t)\right|}-M^{(\alpha)}(t)\\
&=\frac{\left(\widetilde{x}(t)-v(t)\right) \left(f(t,\widetilde{x}(t))
- v^{(\alpha)}(t)\right)}{M(t)}+\left[\frac{M(t)}{a^{\alpha}}
-\frac{|x(t)-v(t)|}{a^{\alpha}}\right]- M^{(\alpha)}(t)\\
&\leq \frac{M(t)M^{\alpha}(t)}{M(t)}+\frac{1}{a^{\alpha}}
\left[ M(t)-|x(t)-v(t)|\right]-M^{(\alpha)}(t)\\
&< 0.
\end{split}
\end{equation*}
On the other hand, we have on $t\in\{ \tau \in A : M(\tau)= 0\}$ that
\begin{equation*}
\begin{split}
\left(|x(t)-v(t)|-M(t)\right)^{(\alpha)}
&=\frac{\left(x(t)-v(t)\right) \left(f(t,\widetilde{x}(t))
+\left(\frac{1}{a^{\alpha}}\widetilde{x}(t)
-\frac{1}{a^{\alpha}}x(t)\right)
-v^{(\alpha)}(t)\right)}{|x(t)-v(t)|} - M^{(\alpha)}(t)\\
&= \frac{\left(x(t)-v(t)\right) \left(f(t,\widetilde{x}(t))
-v^{(\alpha)}(t) \right)}{|x(t)-v(t)|}
-\frac{1}{a^{\alpha}}|x(t)-v(t)|-M^{(\alpha)}(t)\\
&< -M^{(\alpha)}(t)\\
&=0.
\end{split}
\end{equation*}
The last equality follows from Definition~\ref{definition}. If we set
$r(t):= |x(t)-v(t)|-M(t)$, then $r^{(\alpha)}<0$ on $A:= \{t\in [a,b]:r(t)>0\}$.
Moreover, since $(v,M)$ is a tube solution to problem \eqref{pl}
and $x$ satisfies $|x_{0}-v(a)|\leq M(a)$, we know that $r(a)\leq 0$
and Lemma~\ref{le} implies that $A=\emptyset$. Therefore, $x\in \mathrm{T}(v,M)$
and the theorem is proved.
\end{proof}


\section{An Example}
\label{sec:4}

Consider the conformable noninteger order system
\begin{equation}
\label{eq:example}
\begin{cases}
x^{(\frac{1}{2})}=a\frac{\sqrt{t}}{1+t}x^{3}(t)+bx(t)e^{cx(t)},
& t \in[1,2],\\
x(1)=0,
\end{cases}
\end{equation}
where $a,b\in(\infty,0]$ and $c$ is a real constant.
According to Definition~\ref{definition}, $(v,M)\equiv(0,1)$ is a tube solution.
It follows from our Theorem~\ref{thm:mr} that problem \eqref{eq:example}
has a solution $x$ such that $|x(t)|\leq 1$ for every $t\in [1,2]$.


\section*{Acknowledgments}

This research is part of first author's Ph.D. project,
which is carried out at Sidi Bel Abbes University, Algeria.
It was carried out while Bayour was visiting the Department
of Mathematics of University of Aveiro, Portugal, February to April of 2015.
The hospitality of the host institution and the financial support
of University of Chlef, Algeria, are here gratefully acknowledged.
Torres was supported by Portuguese funds through the Center for Research
and Development in Mathematics and Applications (CIDMA) and
the Portuguese Foundation for Science and Technology (FCT),
within project UID/MAT/04106/2013.
The authors are very grateful to three anonymous Referees
for their valuable comments, helpful questions and suggestions.



\bigskip



\begin{thebibliography}{90}

\bibitem{MR3293309}
T. Abdeljawad,
On conformable fractional calculus,
J. Comput. Appl. Math. {\bf 279} (2015), 57--66.
{\tt arXiv:1402.6892}

\bibitem{ref:ref2}
T. Abdeljawad, M. Al Horani\ and\ R. Khalil,
Conformable fractional semigroups of operators,
Journal of Semigroup Theory and Applications {\bf 2015} (2015),
Art. ID~7, 9~pp.
{\tt arXiv:1502.06014}

\bibitem{MR3335759}
D. R. Anderson\ and\ R. I. Avery,
Fractional-order boundary value problem with Sturm-Liouville boundary conditions,
Electron. J. Differential Equations {\bf 2015} (2015), no.~29, 10~pp.
{\tt arXiv:1411.5622}

\bibitem{Anderson}
D. R. Anderson\ and\ D. J. Ulness,
Properties of the Katugampola fractional derivative
with potential application in quantum mechanics,
J. Math. Phys. {\bf 56} (2015), no.~6, 063502, 18~pp.

\bibitem{MR3326681}
H. Batarfi, J. Losada, J. J. Nieto\ and\ W. Shammakh,
Three-Point Boundary Value Problems
for Conformable Fractional Differential Equations,
J. Funct. Spaces {\bf 2015} (2015), Art. ID 706383, 6~pp.

\bibitem{MyID:332}
B. Bayour, A. Hammoudi\ and\ D. F. M. Torres,
Existence of solution to a nonlinear first-order dynamic equation on time scales,
J. Math. Anal. {\bf 7} (2016), no.~1, 31--38.
{\tt arXiv:1512.00909}

\bibitem{MyID:324}
N. Benkhettou, S. Hassani\ and\ D. F. M. Torres,
A conformable fractional calculus on arbitrary time scales,
J. King Saud Univ. Sci. {\bf 28} (2016), no.~1, 93--98.
{\tt arXiv:1505.03134}

\bibitem{MBCT}
M. Bohner\ and\ C. C. Tisdell,
Second order dynamic inclusions,
J. Nonlinear Math. Phys. {\bf 12} (2005), suppl.~2, 36--45.

\bibitem{MR0492115}
P. L. Butzer\ and\ U. Westphal,
An access to fractional differentiation via fractional difference quotients,
in {\it Fractional calculus and its applications
(Proc. Internat. Conf., Univ. New Haven, West Haven, Conn., 1974)},
116--145. Lecture Notes in Math., 457, Springer, Berlin, 1975.

\bibitem{ACG}
A. Cabada, M. R. Grossinho\ and\ F. Minh\'os,
Extremal solutions for third-order nonlinear problems
with upper and lower solutions in reversed order,
Nonlinear Anal. {\bf 62} (2005), no.~6, 1109--1121.

\bibitem{Chung}
W. S. Chung,
Fractional Newton mechanics with conformable fractional derivative,
J. Comput. Appl. Math. {\bf 290} (2015), 150--158.

\bibitem{MF}
M. Frigon,
Boundary and periodic value problems for systems
of nonlinear second order differential equations,
Topol. Methods Nonlinear Anal. {\bf 1} (1993), no.~2, 259--274.

\bibitem{HGF}
M. Frigon\ and\ H. Gilbert,
Existence theorems for systems of third order differential equations,
Dynam. Systems Appl. {\bf 19} (2010), no.~1, 1--23.

\bibitem{HG2}
M. Frigon\ and\ H. Gilbert,
Boundary value problems for systems of second-order dynamic equations
on time scales with $\Delta$-Carath\'eodory functions,
Abstr. Appl. Anal. {\bf 2010} (2010), Art. ID 234015, 26~pp.

\bibitem{HGF2}
M. Frigon\ and\ H. Gilbert,
Systems of first order inclusions on time scales,
Topol. Methods Nonlinear Anal. {\bf 37} (2011), no.~1, 147--163.

\bibitem{HG}
H. Gilbert,
Existence theorems for first-order equations on time scales
with $\Delta$-Carath\'eodory functions,
Adv. Difference Equ. {\bf 2010} (2010), Art. ID 650827, 20~pp.

\bibitem{JGL}
J. R. Graef, L. Kong, F. M. Minh\'{o}s\ and\ J. Fialho,
On the lower and upper solution method
for higher order functional boundary value problems,
Appl. Anal. Discrete Math. {\bf 5} (2011), no.~1, 133--146.

\bibitem{MR1987179}
A. Granas\ and\ J. Dugundji,
{\it Fixed point theory},
Springer Monographs in Mathematics,
Springer, New York, 2003.

\bibitem{Katugampola}
U. N. Katugampola,
A new approach to generalized fractional derivatives,
Bull. Math. Anal. Appl. {\bf 6} (2014), no.~4, 1--15.
{\tt arXiv:1106.0965}

\bibitem{Katugampola:arXiv}
U. N. Katugampola,
A new fractional derivative with classical properties,
preprint, 2014.
{\tt arXiv:1410.6535}

\bibitem{khs}
R. Khalil, M. Al Horani, A. Yousef, M. Sababheh,
A new definition of fractional derivative,
J. Comput. Appl. Math. {\bf 264} (2014), 65--70.

\bibitem{Magin}
R. L. Magin,
{\it Fractional calculus in Bioengineering},
CR in Biomedical Engineering {\bf 32} (2004), no.~1, 1--104.

\bibitem{MR3342452}
M. D. Ortigueira\ and\ J. A. Tenreiro Machado,
What is a fractional derivative?,
J. Comput. Phys. {\bf 293} (2015), 4--13.

\bibitem{MR3351489}
D. Prodanov,
Fractional variation of H\"olderian functions,
Fract. Calc. Appl. Anal. {\bf 18} (2015), no.~3, 580--602.

\bibitem{MRZJ}
M. Ruyun, Z. Jihui\ and\ F. Shengmao,
The method of lower and upper solutions
for fourth-order two-point boundary value problems,
J. Math. Anal. Appl. {\bf 215} (1997), no.~2, 415--422.

\bibitem{MR1347689}
S. G. Samko, A. A. Kilbas\ and\ O. I. Marichev,
{\it Fractional integrals and derivatives},
Gordon and Breach, Yverdon, 1993.

\bibitem{Yang}
X. J. Yang,
{\it Advanced local fractional calculus and its applications},
World Science Publisher, New York, 2012.

\bibitem{MR3204525}
X.-J. Yang, D. Baleanu\ and\ J. A. T. Machado,
Application of the local fractional Fourier series to fractal signals,
in {\it Discontinuity and complexity in nonlinear physical systems},
63--89, Nonlinear Syst. Complex, Springer, Cham, 2014.

\end{thebibliography}
\end{document}